\documentclass[12pt]{amsart}
\usepackage{amsmath}
\usepackage{graphicx}
\usepackage{hyperref}
\usepackage{mathtools}
\usepackage{amsfonts}
\usepackage{amssymb}
\usepackage[T1]{fontenc}
\usepackage{amsthm}
\usepackage{fullpage}
\usepackage{enumitem}
\usepackage[T1]{fontenc}
\usepackage[utf8]{inputenc}
\usepackage{bm}

\newtheorem{theorem}{Theorem}[section]

\newtheorem{lemma}[theorem]{Lemma}

\theoremstyle{definition}

\theoremstyle{remark}
\newtheorem*{remark}{Remark}

\numberwithin{equation}{section}

\title{On Consecutive Non-primitive Elements over Finite Fields}


\author{Bidushi Sharma}
\address{Department of Mathematical Sciences, Tezpur University, Tezpur, Assam, 784028, India}
\email{msp22102@tezu.ac.in}
\author{DHIREN KUMAR BASNET*}
\address{Department of Mathematical Sciences, Tezpur University, Tezpur, Assam, 784028, India}
\email{dbasnet@tezu.ernet.in}

\begin{document}
	\begin{abstract}
Let $q$ be an even prime power and define
\[
\theta_q=\frac{\varphi(q-1)}{q-1}.
\]
In this article, we establish  a bound on $\theta_q$ that guarantees the existence of a pair of consecutive non-primitive elements in $\mathbb{F}_q$, with the exceptions $q=4$ and $q=8$. We first derive a sufficient condition for the existence of such a pair using character sums and then obtain the stated bound by considering several cases according to the least prime divisor of $q-1$.
\end{abstract}
	\maketitle
	
	\section{Introduction}
  Let $q$ be a prime power and define
$\theta_q=\dfrac{\phi(q-1)}{q-1}$,
where $\phi$ denotes Euler's totient function. When $q$ is odd, the set of non-primitive elements can be divided into two disjoint sets, namely, the set of nonzero squares and the set of nonzero non-squares. In \cite{co}, Cohen showed that if $\theta_q < \frac{1}{3}$, or if $\theta_q = \frac{1}{3}$ and $q \notin \{7,13,19,25,37\}$, then the finite field $\mathbb{F}_q$ contains a pair of consecutive elements that are both non-square and non-primitive. This result was originally established for prime fields by Gun, Ramakrishnan, Sahu, and Thangadurai \cite{thg}, who proved that if $\theta_p < \frac{1}{6}$, then $\mathbb{F}_p$ contains a pair of consecutive elements that are both non-square and non-primitive. Jarso and Trudgian \cite{tru} attempted to improve this result by obtaining the bound $\theta_p \leq \frac{1}{4}$ using a sieving method.

    When $q$ is even, every element of $\mathbb{F}_q^*$ is a square. Thus, the elements of $\mathbb{F}_q^*$ can be partitioned into two disjoint sets, the set of primitive elements and the set of non-primitive elements. If $q-1$ is prime, then every non-identity element of $\mathbb{F}_q^*$ is primitive. However, if $q-1$ is composite, then there is at least one non-identity non-primitive element in $\mathbb{F}_q^*$. This leads us to the question of whether there exists a tuple in $\mathbb{F}_q$ consisting of consecutive non-primitive elements. Since the characteristic of $\mathbb{F}_q$ is $2$, any such tuple repeats its elements after an interval of $2$ and therefore has the form
\[
(a,b,a,b,\ldots),
\]
where $a$ and $b$ are non-primitive elements. Thus, the study of such tuples reduces to the study of the existence of a pair of consecutive non-primitive elements, which we shall refer to as an NP pair. This problem was also posed by Cohen in \cite{co}, where he asked whether there exists a bound on $\theta_q$, analogous to the bound in the odd characteristic case, that guarantees the existence of an NP pair. Another observation is that whenever such a pair exists, there must be at least two NP pairs, since $(a,b)$ and $(b,a)$ are both pairs of consecutive non-primitive elements. 

For even $q$, if $\theta_q=1$ or $\theta_q=\frac{q-2}{q-1}$, which occurs only when $q=2$ or $q-1$ is prime, respectively, then every nonzero element of $\mathbb{F}_q^*$ is primitive. Thus, no nonzero non-primitive elements exist in these cases, showing that some restriction on $\theta_q$ is necessary for the existence of an NP pair. Additionally, since $\theta_q$ represents the proportion of primitive elements in $\mathbb{F}_q^*$, $1-\theta_q$ represents the proportion of non-primitive elements in $\mathbb{F}_q^*$. Thus, $\theta_q$ can also be interpreted in terms of the relative abundance of primitive and non-primitive elements. For example, if $\theta_q<\frac{2}{3}$, then the proportion of non-primitive elements exceeds $\frac{1}{3}$, and hence the number of non-primitive elements is more than half the number of primitive elements.

Our main results are summarized in the following theorem.

\begin{theorem}\label{1.1}
Let $q$ be an even prime power and define
\[
\theta_q=\frac{\phi(q-1)}{q-1}.
\]
If $\theta_q<\frac{8}{9}$ then $\mathbb{F}_q$ contains a pair of consecutive non primitive elements except when $q=4, 8.$
\end{theorem}

\begin{remark}
In due course, it becomes evident that $\theta_q=\frac{2}{3}$ and $\theta_q=\frac{6}{7}$ if and only if $q=4$ and $q=8$, respectively. In both cases, no NP pair exists. Also, one might naturally ask what happens when $\theta_q=\frac{8}{9}$. However, it is readily seen that there is no prime power $q$ for which $\theta_q=\frac{8}{9}$.
\end{remark}

We have another important theorem that plays a key role in the proof of Theorem~\ref{1.1}. For any prime divisor $\ell$ of $q-1$, the elements $0$ and $1$ are always trivial $\ell$th powers. We shall therefore refer to the elements of $\mathbb{F}_q\setminus\{0,1\}$ that are $\ell$th powers as non-trivial $\ell$th powers, abbreviated as \emph{NT$\ell$}. Accordingly, a pair of consecutive non-trivial $\ell$th powers will be called an \emph{NT$\ell$ pair}. Since $0$ is neither primitive nor non-primitive, every NT$\ell$ pair is necessarily an NP pair.

\begin{theorem}\label{T1.2}
Let $q$ be an even prime power and let $\ell$ be the least prime divisor of $q-1$. If $\theta_q<\frac{8}{9}$ then $\mathbb{F}_q$ contains an $\mathrm{NT}\ell$ pair, except when $q=4,8,16$.
\end{theorem}
\section{Consecutive $\ell$th Powers}
Let $q$ be an odd prime power, and let $\ell$ be a prime divisor of $q-1$. Let $\chi\in\widehat{\mathbb{F}_q^*}$ be a multiplicative character of order $\ell$, extended to $\mathbb{F}_q$ by setting $\chi(0)=0$. Define
\[
M_{\ell}(q)=\left\{\beta\in\mathbb{F}_q \mid \beta\text{ and }\beta+1\text{ are both non-trivial $\ell$th powers}\right\}.
\]
Our first aim is to obtain a lower bound for $\#M_{\ell}(q)$. 

For $\beta\in\mathbb{F}_q$, the indicator function of the non-trivial $\ell$th powers in $\mathbb{F}_q$ is given by
\[
g_{\ell}(\beta)=\frac{1}{\ell}\sum_{i=0}^{\ell-1}\chi^i(\beta)-\delta_1(\beta),
\]
where
\[
\delta_1(\beta)=
\begin{cases}
1, & \beta=1,\\
0, & \beta\neq 1.
\end{cases}
\]
Consequently,
\[
\#M_{\ell}(q)
=
\sum_{\beta\in\mathbb{F}_q\setminus\{0,1\}}
g_{\ell}(\beta)g_{\ell}(\beta+1).
\]
Upon expanding the sum and using the fact that $\beta$ and $\beta+1$ can never simultaneously equal $1$, together with $\delta_1(\beta)=\delta_1(\beta+1)=0$ for all $\beta\in\mathbb{F}_q\setminus\{0,1\}$, we obtain
\begin{align*}
\#M_{\ell}(q)\ell^2
&=
\sum_{i,j=0}^{\ell-1}
\sum_{\beta\in\mathbb{F}_q\setminus\{0,1\}}
\chi^i(\beta)\chi^j(\beta+1)\\
&=q-2+C_{1}+C_{2}+C_{3},
\end{align*}
where 
\[
C_1=
\sum_{i=1}^{\ell-1}
\sum_{\beta\in\mathbb{F}_q\setminus\{0,1\}}
\chi^i(\beta)
=-(\ell-1),
\qquad
C_2=
\sum_{j=1}^{\ell-1}
\sum_{\beta\in\mathbb{F}_q\setminus\{0,1\}}
\chi^j(\beta+1)
=-(\ell-1)
\]
and
\[
\begin{aligned}
|C_3|
&=
\left|
\sum_{i,j=1}^{\ell-1}
\sum_{\beta\in\mathbb{F}_q\setminus\{0,1\}}
\chi^i(\beta)\chi^j(\beta+1)
\right|\leq (\ell-1)^2\sqrt{q},
\end{aligned}
\]
where the last inequality follows from Corollary 2.3 \cite{wang}.

Thus, we obtain the following sufficient condition for the existence of an $\mathrm{NT}\ell$ pair.
\begin{theorem}\label{T2.1}
Let $q$ be an even prime power and let $\ell$ be a prime divisor of $q-1$. Then $\mathbb{F}_q$ contains an $\mathrm{NT}\ell$ pair whenever
\[
q>9\ell^4.
\]
\end{theorem}
As an application of Theorem~\ref{T2.1}, we obtain explicit sufficient conditions for the existence of $\mathrm{NT}\ell$ pairs. In particular, if $3$ divides $q-1$, then every even prime power $q>9\cdot3^4=729$ contains an $\mathrm{NT}3$ pair. Similarly, if $5$, $7$, or $11$ divides $q-1$, then $\mathbb{F}_q$ contains an $\mathrm{NT}5$, $\mathrm{NT}7$, or $\mathrm{NT}11$ pair, respectively, whenever
\[
q>9\cdot5^4=5625,\qquad
q>9\cdot7^4=21609,\qquad
q>9\cdot11^4=131769.
\]
Here, we do not require $\ell$ to be the least prime divisor of $q-1$.
\begin{remark}
The following two observations will be useful in the proofs of the results that follow.
\begin{enumerate}
    \item If $q=2^k$, then $3$ is the least prime divisor of $q-1$ if and only if $k$ is even.
    \item If $q=2^k$ with $k$ odd, then $7$ is the least prime divisor of $q-1$ if and only if $3\mid k$.
\end{enumerate}
\end{remark}

\begin{lemma}\label{L2.2}
Let $q=2^k$, where $k>2$ is even. Then $\theta_q<\frac{2}{3}$, and $\mathbb{F}_q$ always contains an $\mathrm{NT}3$ pair, except when $q=16$.
\end{lemma}

\begin{proof}
Since $k$ is even, $3$ is the least prime divisor of $q-1$. Thus,
\[
\theta_q
=
\left(1-\frac{1}{3}\right)
\prod_{\substack{p\mid q-1\\p\neq 3}}
\left(1-\frac{1}{p}\right)
\leq
1-\frac{1}{3}
=
\frac{2}{3}.
\]
We claim that the inequality is strict. Indeed, suppose that $\theta_q=\frac{2}{3}$ for some $q$. Then $3$ must be the only prime divisor of $q-1$. Since $k>2$ is even, we may write $k=2m$ with $m>1$. Hence,
\[
q-1=2^{2m}-1=(2^m-1)(2^m+1).
\]
If $m$ is odd, then $2^m-1>3$ and it always has a prime divisor other than $3$. If $m$ is even, then $2^m+1>3$ and is always a prime. In either case, $3$ is not the only prime divisor of $q-1$, which is a contradiction. Therefore,
\[
\theta_q<\frac{2}{3}.
\]

By Theorem~\ref{T2.1}, it remains only to verify the existence of an $\mathrm{NT}3$ pair for even prime powers satisfying $q\leq 9\cdot 3^4=729$.
The only cases that remain to be checked are
\[
q=2^4,\qquad q=2^6,\qquad q=2^8.
\]
For this purpose, we use a simple SageMath program to compute $\#M_3(q)$ and to identify all $\mathrm{NT}3$ pairs for the remaining cases. The corresponding values are presented in Table~\ref{tab:nt3}.

\begin{table}[h]
\centering
\scriptsize
\begin{tabular}{c|c|c|p{9cm}}
$q$ & $\#M_3(q)$ & $\theta_q$ & $\mathrm{NT}3$ pairs \\ \hline
$2^4$ & $0$ & $0.5333$ & None \\[2mm]

$2^6$ & $8$ & $0.5714$ &
$\left(a^5+a^4+a^2+1,a^5+a^4+a^2\right)$,
$\left(a^4+a^2+a+1,a^4+a^2+a\right)$,
$\left(a^3+a^2+a,a^3+a^2+a+1\right)$,
$\left(a^5+a^4+a^2,a^5+a^4+a^2+1\right)$,
$\left(a^5+a,a^5+a+1\right)$,
$\left(a^3+a^2+a+1,a^3+a^2+a\right)$,
$\left(a^5+a+1,a^5+a\right)$,
$\left(a^4+a^2+a,a^4+a^2+a+1\right)$ \\[3mm]

$2^8$ & $24$ & $0.5019$ &
$\left(a^5+a^4+a^3+a,a^5+a^4+a^3+a+1\right)$,
$\left(a^5+a^2+a,a^5+a^2+a+1\right)$,
$\left(a^5+a^3+a^2+1,a^5+a^3+a^2\right)$,
$\left(a^6+a^5,a^6+a^5+1\right)$,
$\left(a^5+a^2+a+1,a^5+a^2+a\right)$,
$\left(a^5+a^2+1,a^5+a^2\right)$,
$\left(a^7+a^5+a^4+a^3+1,a^7+a^5+a^4+a^3\right)$,
$\left(a^6+a^5+1,a^6+a^5\right)$,
$\left(a^6+a^5+a^2+1,a^6+a^5+a^2\right)$,
$\left(a^7+a^6+a^3+a^2+a,a^7+a^6+a^3+a^2+a+1\right)$,
$\left(a^5+a^4+a^3+a+1,a^5+a^4+a^3+a\right)$,
$\left(a^7+a^6+a^2+1,a^7+a^6+a^2\right)$,
$\left(a^7+a^5+a^4+a^3,a^7+a^5+a^4+a^3+1\right)$,
$\left(a^7+a^5+a^3+1,a^7+a^5+a^3\right)$,
$\left(a^7+a^5+a^3,a^7+a^5+a^3+1\right)$,
$\left(a^7+a^6+a^2,a^7+a^6+a^2+1\right)$,
$\left(a^6+a^4+a^2+a+1,a^6+a^4+a^2+a\right)$,
$\left(a^6+a^5+a^2,a^6+a^5+a^2+1\right)$,
$\left(a^6+a^4+a^2+a,a^6+a^4+a^2+a+1\right)$,
$\left(a^7+a^3+a,a^7+a^3+a+1\right)$,
$\left(a^5+a^2,a^5+a^2+1\right)$,
$\left(a^7+a^3+a+1,a^7+a^3+a\right)$,
$\left(a^5+a^3+a^2,a^5+a^3+a^2+1\right)$,
$\left(a^7+a^6+a^3+a^2+a+1,a^7+a^6+a^3+a^2+a\right)$
\end{tabular}
\caption{The values of $\#M_3(q)$ and $\theta_q$, together with the corresponding $\mathrm{NT}3$ pairs.}
\label{tab:nt3}
\end{table}
This completes the verification of the remaining cases.
\end{proof}
\begin{lemma}\label{L2.3}
Let $q=2^k$, where $k>3$ is an odd multiple of $3$. Then $\theta_q<\frac{6}{7}$, and $\mathbb{F}_q$ contains an $\mathrm{NT}7$ pair.
\end{lemma}
\begin{proof}
By hypothesis, $k=3(2m+1)$ for some $m> 1$. Since $k$ is an odd multiple of $3$, $7$ is the least prime divisor of $q-1$. Hence,
\[
\theta_q
\leq\frac{6}{7}.
\]
We claim that the inequality is strict. Suppose, to the contrary, that $\theta_q=\frac{6}{7}$ for some $q$.
Then $7$ must be the only prime divisor of $q-1$. Put $n=2m+1$, so that $n>1$ is odd. We have
\[
q-1=2^{3n}-1
=(2^n-1)(2^{2n}+2^n+1).
\]

Now suppose that $n$ is not divisible by $3$, then $2^n\not\equiv 1\pmod 7$, and thus $2^n-1$ has a prime divisor different from $7$, a contradiction.

Finally, suppose that $3\mid n$. Then 
\[
2^{2n}+2^n+1\equiv 1+1+1\equiv3\pmod 7,
\]
 which is again a contradiction.

Thus, $\theta_q<\frac{6}{7}$.
By Theorem~\ref{T2.1}, it remains only to verify the existence of an $\mathrm{NT}7$ pair for even prime powers satisfying $q\leq 9\cdot7^4=21609$. 

Under the present hypotheses, the only value that remains to be checked is $q=2^9$. Using SageMath we obtain the following data.
\begin{table}[h]
\centering
\scriptsize
\begin{tabular}{c|c|c|p{8cm}}
$q$ & $\#M_7(q)$ & $\theta_q$ & $\mathrm{NT}7$ pairs \\ \hline
$2^9$ & $18$ & $0.8454$ &
$\left(a^8+a^6+a^5+a^3+a,\,
a^8+a^6+a^5+a^3+a+1\right)$, \\
&&&
$\left(a^8+a^6+a^2+a+1,\,
a^8+a^6+a^2+a\right)$, \\
&&&
$\left(a^5+a^3+a,\,
a^5+a^3+a+1\right)$, \\
&&&
$\left(a^6+a^4+a^3+1,\,
a^6+a^4+a^3\right)$, \\
&&&
$\left(a^8+a^6+a^2+a,\,
a^8+a^6+a^2+a+1\right)$, \\
&&&
$\left(a^6+a^5+a^2+a,\,
a^6+a^5+a^2+a+1\right)$, \\
&&&
$\left(a^8+a^7+a^6+a^3+1,\,
a^8+a^7+a^6+a^3\right)$, \\
&&&
$\left(a^8+a^6+a^5+a^3+a+1,\,
a^8+a^6+a^5+a^3+a\right)$, \\
&&&
$\left(a^6+a^4+a^3,\,
a^6+a^4+a^3+1\right)$, \\
&&&
$\left(a^8+a^6+a^5+a^4+a,\,
a^8+a^6+a^5+a^4+a+1\right)$, \\
&&&
$\left(a^7+a^5+a^4+a^3+a^2+a,\,
a^7+a^5+a^4+a^3+a^2+a+1\right)$, \\
&&&
$\left(a^7+a^5+a^4+a^3+a^2+a+1,\,
a^7+a^5+a^4+a^3+a^2+a\right)$, \\
&&&
$\left(a^8+a^6+a^5+a^4+a+1,\,
a^8+a^6+a^5+a^4+a\right)$, \\
&&&
$\left(a^5+a^4+a^3+a^2,\,
a^5+a^4+a^3+a^2+1\right)$, \\
&&&
$\left(a^6+a^5+a^2+a+1,\,
a^6+a^5+a^2+a\right)$, \\
&&&
$\left(a^5+a^4+a^3+a^2+1,\,
a^5+a^4+a^3+a^2\right)$, \\
&&&
$\left(a^5+a^3+a+1,\,
a^5+a^3+a\right)$, \\
&&&
$\left(a^8+a^7+a^6+a^3,\,
a^8+a^7+a^6+a^3+1\right)$.
\end{tabular}
\caption{The $\mathrm{NT}7$ pairs in $\mathbb{F}_{2^9}$.}
\label{tab:nt7}
\end{table}

This completes our proof.
\end{proof}
\begin{remark}
We have completely covered the cases in which the least prime divisor of $q-1$ is $3$ or $7$ and $q-1$ is composite. Before proceeding further, consider the case $q=2^k$, where $k>1$ is odd and $3\nmid k$. In this case, the possible values of the least prime divisor of $q-1$ not exceeding $100$ are $23,\ 31,\ 47,\ 71,\ 89$.
To understand this, for an odd prime $p$, define $\operatorname{ord}_p(2)$ to be the smallest positive integer $k$ such that $2^k\equiv1\pmod p$.
If $2\mid k$ or $3\mid k$, then the least odd prime divisor of $q-1$ is necessarily either $3$ or $7$. Thus, when $k$ is odd and $3\nmid k$, we need only consider primes mentioned above. The following table lists $\operatorname{ord}_p(2)$ for all primes $p\leq100$ with $p\neq3,7$.

\begin{table}[h]
\centering
\small
\begin{tabular}{c|c||c|c}
\hline
$p$ & $\operatorname{ord}_p(2)$ & $p$ & $\operatorname{ord}_p(2)$ \\
\hline
$5$  & $4$  & $53$ & $52$ \\
$11$ & $10$ & $59$ & $58$ \\
$13$ & $12$ & $61$ & $60$ \\
$17$ & $8$  & $67$ & $66$ \\
$19$ & $18$ & $71$ & $35$ \\
$23$ & $11$ & $73$ & $9$  \\
$29$ & $28$ & $79$ & $39$ \\
$31$ & $5$  & $83$ & $82$ \\
$37$ & $36$ & $89$ & $11$ \\
$41$ & $20$ & $97$ & $48$ \\
$43$ & $14$ &      &      \\
$47$ & $23$ &      &      \\
\hline
\end{tabular}
\caption{Values of $\operatorname{ord}_p(2)$ for primes $p\leq100$, $p\neq3,7$.}
\label{tab:ord2}
\end{table}
\end{remark}
\begin{theorem}\label{L2.4}
Let $q=2^k$, where $k>1$ is odd and $3\nmid k$, and let $\ell$ be the least prime divisor of $q-1$. Then $\mathbb{F}_q$ contains an $\mathrm{NT}\ell$ pair whenever
\[
\theta_q<\frac{8}{9}.
\]
\end{theorem}
\begin{proof}
By hypothesis, the least prime divisor $\ell$ of $q-1$ satisfies $\ell\geq23$. If $q-1$ has at most four distinct prime divisors, then
\[
\theta_q
\geq
\left(1-\frac{1}{23}\right)
\left(1-\frac{1}{31}\right)
\left(1-\frac{1}{47}\right)
\left(1-\frac{1}{71}\right)
\approx 0.8931
>
\frac{8}{9}.
\]
Therefore, then $q-1$ must have at least five distinct prime divisors. Since every prime divisor of $q-1$ is at least $\ell$, we have
\[
q-1\geq \ell^5>9\ell^4.
\]
Hence, by Theorem~\ref{T2.1}, $\mathbb{F}_q$ contains an $\mathrm{NT}\ell$ pair. 
\end{proof}

\begin{remark}
In Lemmas~\ref{L2.2}, \ref{L2.3}, and~\ref{L2.4}, the cases $q=2$, $2^2$, and $2^3$ were excluded. The corresponding values of $\theta_q$ are $1$, $\frac{2}{3}$, and $\frac{6}{7}$, respectively. In each of these cases, $q-1$ is either $1$ or a prime, and hence no $\mathrm{NT}\ell$ pair exists. This completes the proof of Theorem~\ref{T1.2}.
\end{remark}
\section{Consecutive Non-Primitive Elements}
In this section, we prove Theorem~\ref{1.1}. Lemmas~\ref{L2.2}, \ref{L2.3}, and~\ref{L2.4} together cover all cases except $\mathbb{F}_{16}$. The remaining case is settled by the following lemma, which completes the proof of our main result.
\begin{lemma}
The field $\mathbb{F}_{16}$ contains exactly two NP pairs.
\end{lemma}
\begin{proof}
We have
\[
\mathbb{F}_{16}\cong \mathbb{F}_2[x]/\langle x^4+x+1\rangle.
\]
Let $\alpha$ be a root of $x^4+x+1$. Then $\mathbb{F}_{16}^*=\langle\alpha\rangle$ is cyclic of order $15$.

The group $\mathbb{F}_{16}^*$ has a unique subgroup of order $3$ and a unique subgroup of order $5$, which we denote by $H_3$ and $H_5$, respectively. Thus,
\[
H_3=\{1,\alpha^5,\alpha^{10}\}
\]
and
\[
H_5=\{1,\alpha^3,\alpha^6,\alpha^9,\alpha^{12}\}.
\]
Since
\[
\alpha^5=\alpha^2+\alpha
\qquad\text{and}\qquad
\alpha^{10}=\alpha^2+\alpha+1,
\]
the elements $\alpha^5$ and $\alpha^{10}$ are non-primitive and consecutive. Hence,
\[
(\alpha^5,\alpha^{10})
\]
is an NP pair.

Moreover, for every $k\in H_5\setminus\{1\}$, we have $k+1\notin H_5$. Thus, no additional NP pair arises from the elements of $H_5$. Therefore, $\mathbb{F}_{16}$ contains exactly two NP pairs, namely
\[
(\alpha^5,\alpha^{10})
\quad\text{and}\quad
(\alpha^{10},\alpha^5).
\]
This completes the proof.
\end{proof}

    \end{document}